\documentclass[a4paper,12pt]{amsart}
\usepackage{fouriernc}
\usepackage[T1]{fontenc}

\usepackage{graphicx}
\usepackage{amsmath}
\usepackage{amssymb}
\usepackage[active]{srcltx}
\usepackage{hyperref}
\usepackage{bbm}
\usepackage{enumerate}
\usepackage{tikz}

\usepackage{geometry}
\newtheorem{thm}{Theorem}%
\newtheorem{lem}[thm]{Lemma}%[section]%
\newtheorem{prop}[thm]{Proposition}%[section]%
\newtheorem*{thm-non}{Theorem}
\theoremstyle{definition}

\theoremstyle{remark}
\theoremstyle{plain}

\numberwithin{equation}{section}

\def\CC{{\mathbb C}}

\def\EE{{\mathbb E}}

\def\HH{{\mathbb H}}

\def\NN{{\mathbb N}}

\def\RR{{\mathbb R}}
\def\TT{{\mathbb T}}

\def\ZZ{{\mathbb Z}}

\def\one{{\mathbbm{1}}}

\def\S{\operatorname{S{}}}

\def\vecell{{\text{\boldmath$\ell$}}}
\def\vecm{{\text{\boldmath$m$}}}
\def\vecn{{\text{\boldmath$n$}}}

\def\vecs{{\text{\boldmath$s$}}}

\def\vecx{{\text{\boldmath$x$}}}

\def\vecy{{\text{\boldmath$y$}}}

\def\vecgamma{{\text{\boldmath$\gamma$}}}

\def\veceta{{\text{\boldmath$\eta$}}}

\def\vectheta{{\text{\boldmath$\theta$}}}

\def\vecnull{{\text{\boldmath$0$}}}

\def\scrB{{\mathcal B}}

\def\scrF{{\mathcal F}}

\def\scrH{{\mathcal H}}

\def\scrM{{\mathcal M}}

\def\scrT{{\mathcal T}}

\def\e{\mathrm{e}}
\def\i{\mathrm{i}}

\def\id{\operatorname{id}}

\def\C{\operatorname{C{}}}

\def\L{\operatorname{L{}}}

\def\Sp{\operatorname{Sp}}

\def\SL{\operatorname{SL}}

\def\PSL{\operatorname{PSL}}

\def\tr{\operatorname{tr}}

\def\sign{\operatorname{sign}}

\def\Stab{\operatorname{Stab}}

\def\Vol{\operatorname{Vol}}

\def\trans{\,^\mathrm{t}\!}

\def\mcg{\mathrm{MCG_g}}

\newcommand{\counting}{N_{f,L,\tau}}

\title[The moduli space of twisted Laplacians]{The moduli space of twisted Laplacians\\ and random matrix theory}

\author{Jens Marklof}
\address{Jens Marklof, School of Mathematics, University of Bristol, Bristol BS8 1UG, U.K.\newline \rule[0ex]{0ex}{0ex} \hspace{8pt}{\tt j.marklof@bristol.ac.uk}}
\author{Laura Monk}
\address{Laura Monk, School of Mathematics, University of Bristol, Bristol BS8 1UG, U.K.\newline \rule[0ex]{0ex}{0ex} \hspace{8pt}{\tt laura.monk@bristol.ac.uk}}
\date{15 July 2024/30 September 2024}
\thanks{Research supported by EPSRC grant EP/W007010/1. Data supporting this study are included within the article. MSC (2020): 81Q50 (Primary); 11F72, 30F60, 58J50 (Secondary)}

\begin{document}

\begin{abstract}
Rudnick recently proved that the spectral number variance for the Laplacian of a large compact hyperbolic surface converges, in a certain scaling limit and when averaged with respect to the Weil-Petersson measure on moduli space, to the number variance of the Gaussian Orthogonal Ensemble of random matrix theory. In this article we extend Rudnick's approach to show convergence to the Gaussian Unitary Ensemble for twisted Laplacians which break time-reversal symmetry, and to the Gaussian Symplectic Ensemble for Dirac operators. This addresses a question of Naud, who obtained analogous results for twisted Laplacians on high degree random covers of a fixed compact surface. 
\end{abstract}

\maketitle

\tableofcontents

\section{Introduction}\label{secIntro}

One of the long-standing conjectures in quantum chaos is that the spectral statistics of quantum
systems with chaotic classical limits are governed by random matrix theory \cite{BGS}. Despite
convincing heuristics \cite{Haake,SieberRichter}, there is currently not a single example where this
phenomenon can be established rigorously. What is more, counter examples are known in arithmetic
settings \cite{BGS2,BolteSteiner,LuoSarnak}, so the conjecture can only be assumed to hold
generically. Rudnick recently pointed out that random matrix statistics emerge in a certain scaling
regime for the spectrum of the Laplacian on hyperbolic surfaces of large genus \cite{Rudnick}, if
one averages over the moduli space of such surfaces with respect to the Weil-Petersson measure. (See
also \cite{AurichSteiner} for an early numerical study of spectral statistics on average over moduli
space.) These statistics are given by the Gaussian Orthogonal Ensemble (GOE), which is consistent
with Bohigas, Giannoni and Schmit's conjecture \cite{BGS}, since the underlying classical dynamics
(the geodesic flow on the hyperbolic surfaces) is time-reversal invariant. Twisted Laplacians, which
model the Aharonov-Bohm effect on hyperbolic surfaces, break time-reversal symmetry, and their
statistics are expected to follow the Gaussian Unitary Ensemble (GUE) \cite{BGS}. The purpose of
this paper is to construct the moduli space of such twisted Laplacians, which allows us to extend
Rudnick's result to these operators, and prove compatibility with GUE statistics in the same scaling
regime for twisted Laplacians on large genus hyperbolic surfaces with random flux. Analogous results,
both for GUE and for GOE statistics, were established by Naud \cite{Naud} in the case of random
covers of a fixed compact hyperbolic surface with and without Aharonov-Bohm flux. The mechanism that
produces the difference between GUE and GOE through the application of semiclassical trace formulas
was first pointed out by Berry \cite{Berry} in his diagonal approximation argument. The idea of
averaging over Aharonov-Bohm flux lines to obtain GUE statistics goes back to a paper of Fishman and
Keating \cite{Fishman}, which investigates the case of planar quantum billiards.

The framework developed in this paper also produces analogous results for Dirac operators, where we
establish convergence to the Gaussian Symplectic Ensemble (GSE) in the presence of time-reversal
symmetry, and otherwise to GUE.

Rudnick and Wigman have refined the results of \cite{Rudnick} by proving a central limit theorem for
the smoothed spectral counting function \cite{RudnickWigman23} and asymptotic almost sure
convergence of the spectral variance to GOE \cite{RudnickWigman24}, in the same scaling limit used
in \cite{Rudnick} and the present paper. We will leave these and other possible extensions to the
GUE and GSE setting to future study.

The set-up for twisted Laplacians is as follows; cf.~\cite{PhillipsSarnak} for more details and
background. We represent a compact oriented hyperbolic surface of genus $g$ as
$X=\Gamma\backslash\HH$, where $\Gamma\subset\PSL(2,\RR)$ is a co-compact Fuchsian group acting by
M\"obius transformations on the complex upper half plane $\HH=\RR+\i\RR_{>0}$. The Laplacian
$\Delta=y^2(\partial_x^2+\partial_y^2)$ acts on twice differentiable functions $\varphi: \HH\to\CC$
parametrised by $z=x+\i y$. Given a unitary character $\chi:\Gamma\to\CC^\times$ (with $\CC^\times$
the multiplicative group of complex numbers), we restrict to the Hilbert space $\scrH_\chi$ of
functions satisfying the relation
\begin{equation}
\varphi(\eta\cdot z) = \chi(\eta) \, \varphi(z) \quad \text{for all $\eta\in\Gamma$,}
\end{equation}
and with inner product 
\begin{equation}
\langle \varphi_1,\varphi_2 \rangle = \int_{\Gamma\backslash\HH} \overline{\varphi_1(z)}\; \varphi_2(z)\; \frac{dx\,dy}{y^2}.
\end{equation}
The restriction of $\Delta$ to $\scrH_\chi$ is essentially self-adjoint; we denote its self-adjoint extension by $\Delta_\chi$. The operator $-\Delta_\chi$ has non-negative discrete spectrum 
\begin{equation}
0\leq \lambda_0 \leq \lambda_1 \leq \ldots \to\infty
\end{equation}
with only accumulation point at $\infty$. If $\chi$ is not a real-valued function, we will refer
to~$\Delta_\chi$ as a twisted Laplacian. This operator models the Aharonov-Bohm effect of a free
quantum particle on the surface subject to $2g$ magnetic flux lines of strength
$\theta_1,\ldots,\theta_{2g}$. Explicitly,
$\chi(\eta) = \chi_\vectheta(\eta) = \exp(2\pi\i (\theta_1[\eta]_1 +\ldots+\theta_{2g}[\eta]_{2g}))$
with $[\eta]_j$ the topological winding number around the $j$th flux line of the closed geodesic
corresponding to $\eta$. Provided that \mbox{$\chi\neq \overline\chi$}, the twisted Laplacian
violates time-reversal symmetry, and one expects the spectral statistics to be described by the
unitary (GUE) rather than the orthogonal (GOE) random matrix ensemble.

It is convenient to define $r_j = \sqrt{\lambda_j - 1/4}$ with the branch of the square root
chosen so that $\arg r_j \in [0,\pi/2]$. That is, $r_j\in\RR_{\geq 0}$ for $\lambda_j\geq
1/4$ and $r_j\in \i (0,1/2]$ for $\lambda_j< 1/4$.
The main tool to analyse the distribution of $r_j$ is the Selberg trace formula, which relates spectral to geometric data. For any even test function 
\begin{equation}
h(r) = \int_\RR \hat h(u)\, \e^{\i u r} du ,
\end{equation}
with $\hat h\in\C^\infty(\RR)$ compactly supported, we have
\begin{equation}\label{STF}
 \sum_{j=0}^\infty h(r_j)   = (g-1) \int_\RR h(r) r \tanh(\pi r) \, dr + 
      \sum_{n=1}^{\infty} \sum_{\{\eta\}} \chi(\eta)^n\,
      \frac{\ell_\eta \hat{h}(n \ell_\eta)}{2 \sinh(n \ell_\eta/2)}.
  \end{equation}
Here the second sum is over all conjugacy classes $\{\eta\}=\{ \delta\eta\delta^{-1} : \delta\in\Gamma\}$ of primitive elements $\eta\in\Gamma$, $\eta\neq 1$ ($\eta$ is called primitive in $\Gamma$, if $\eta=\delta^n$ for some $\delta\in\Gamma$ implies $n=\pm 1$).
The set of primitive conjugacy classes $\{\eta\}$ is in one-to-one correspondence with primitive closed oriented geodesics on $X$, and their lengths are denoted by $\ell_\eta$. 

Following Rudnick \cite{Rudnick}, we consider the distribution of the smoothed counting function of the $r_j$,
\begin{equation}
\counting = \sum_{j=0}^\infty h_{L,\tau}(r_j) ,\qquad 
h_{L,\tau}(r) = f(L(r-\tau)) + f(L(r+\tau)) ,
\end{equation}
with $L,\tau>0$ and assuming $f$ is even with a smooth and compactly supported Fourier transform
\begin{equation}
\hat f(u) = \frac{1}{2\pi}\int_\RR f(r)\, \e^{-\i u r} dr .
\end{equation}
This ensures our assumptions on $h=h_{L,\tau}$ are satisfied. 
More precisely, we are interested in the statistical distribution of $\counting$ for random choices of $(X, \vectheta)$.

The main new construction presented in this paper is the moduli space $\widetilde\scrM_{g}$ of the
pair $(X,\vectheta)$ for a given genus $g$, together with a family of natural probability measures
$\mu_{g,q}$ on $\widetilde\scrM_{g}$. These measures may be viewed as an extension of the classical
Weil-Petersson probability measure $\mu_g^{\mathrm{WP}}$ on the moduli space $\scrM_g$ of hyperbolic
surfaces $X$ of genus $g$ considered in~\cite{Rudnick}.  The index
$q \in \mathbb{N} \cup \{ \infty\}$ indicates whether the measure on $\vectheta$ is Lebesgue
($q=\infty$) or discrete, supported on rationals with denominator $q\in\NN$. The cases $q=1$ and
$q=2$ correspond to the trivial resp.\ general real characters $\chi(\eta)\in\{\pm 1\}$ (where we
expect GOE), and $3\leq q\leq \infty$ to complex characters (where we expect GUE).

We will now state our principal result and postpone the precise definition of the extended moduli
space $\widetilde\scrM_{g}$ and measures $\mu_{g,q}$ to the remainder of the paper.  We denote by
$\EE_g^{\mathrm{WP}}$ and $\EE_{g,q}$ the expected values defined by
$$\EE_g^{\mathrm{WP}} [\ldots] = \int_{\scrM_g} [\ldots]\, d\mu_g^{\mathrm{WP}}, \qquad \EE_{g,q} [\ldots] = \int_{\widetilde\scrM_{g}} [\ldots]\, d\mu_{g,q}.$$ 

Let $\Sigma_{g,q}^2(L,\tau; f)$ denote the variance of the counting function $\counting$, defined as
\begin{equation}\label{Sigmagq}
\Sigma_{g,q}^2(L,\tau; f) = \EE_{g,q} \bigg[ \big| \counting - \EE_{g,q} [\counting] \big|^2 \bigg].
\end{equation}

\begin{thm}\label{thm1}
Fix $\tau>0$. For $f:\RR\to\CC$ even with smooth, compactly supported Fourier transform, we have, for $L\geq 2$, 
\begin{equation}
\lim_{g\to\infty} \Sigma_{g,q}^2(L,\tau; f) 
= 
\begin{cases}
\Sigma_{\mathrm{GOE}}^2(f) +O(L^{-2}\log L) & (q=1,2)\\
\Sigma_{\mathrm{GUE}}^2(f) +O(L^{-2}\log L) & (3\leq q \leq \infty) ,
\end{cases}
\end{equation}
with the random matrix densities
\begin{equation}
\Sigma_{\mathrm{GOE}}^2(f) = \int_\RR 2|x| \, |\hat f(x)|^2 dx, \qquad \Sigma_{\mathrm{GUE}}^2(f) = \int_\RR |x| \, |\hat f(x)|^2 dx.
\end{equation}
\end{thm}

The analogous result for Dirac operators is given in Section \ref{secDirac}.\\

The plan for this paper is as follows. In Section \ref{secAffine} we define a fiber bundle over the Teich\-m\"uller space of hyperbolic surfaces $X$ of genus $g$, and a natural action of the mapping class group on this extended Teichm\"uller space. The fiber at each point $X$ is a torus which parametrises the twist parameter of the Laplacian. By quotienting the extended Teichm\"uller space by the mapping class group, we obtain a natural moduli space of twisted Laplacians. We then define invariant probability measures on this space in Section \ref{secInvariant}, which model continuous and discrete random twists. The projection of these measures to the base moduli space is the classical Weil-Petersson measure. Although not essential for subsequent discussion, we also introduce the Fourier transform on the fiber, and describe the action of the mapping class group on Fourier space. In Section~\ref{secLength} we focus on a particular class of functions on the extended moduli space. These are length-holonomy functions which generalise Mirzakhani's geometric functions and appear on the geometric side of the Selberg trace formula for twisted Laplacians. The precise parametrisation of twisted Laplacians in terms of the extended Teichm\"uller space is explained in Section \ref{secTwisted}. Section \ref{secMain} states and proves the main lemmas that are required for the proof of Theorem \ref{thm1} in Section \ref{secProof}. We extend our approach to Dirac operators with and without time-reversal symmetry in Section \ref{secDirac}, and prove convergence to GSE resp.~GUE statistics. \\

\textbf{Acknowledgements.}
We thank Zeev Rudnick for very helpful comments on the first drafts of this paper.

\section{The affine mapping class group}\label{secAffine}

Let $S_g$ be a given compact, connected, oriented, smooth surface of genus $g\geq 2$, which we shall refer to as the base surface.  The Teichmüller space $\mathcal{T}_g$ is the set of equivalence classes of marked hyperbolic surfaces $(X, \phi)$, where (i) $X$ is a hyperbolic surface, (ii) $\phi$ is a diffeomorphism $S_g \rightarrow X$ (the marking), and (iii) $(X_1, \phi_1)$, $(X_2, \phi_2)$ are called {\em equivalent} if there exists an isometry $m : X_1 \rightarrow X_2$  such that $m \circ \phi_1$ is isotopic to $\phi_2$.

The mapping class group $\mcg$ of the base surface $S_g$ is defined as the group of positive (i.e., orientation-preserving) diffeomorphisms $\psi : S_g \rightarrow S_g$ modulo isotopies.
The mapping class group acts naturally on the Teichmüller space $\mathcal{T}_g$ by 
\begin{equation}
  \psi \cdot (X, \phi) = (X, \phi \circ \psi^{-1}).
\end{equation}
This action allows us to ``forget'' the marking, and so the quotient $\mathcal{M}_g:=\mcg\backslash\mathcal{T}_g$ yields the moduli space of hyperbolic surfaces $X$ of genus $g$ modulo isometries.

The first homology $H_1(S_g; \ZZ)$ of the base surface $S_g$ is the
abelianisation of its fundamental group $\pi_1(S_g)$. Let us fix a basis for $H_1(S_g; \ZZ)$
by choosing the family $a_1,\ldots,a_g$, $b_1,\ldots,b_g$ of fundamental oriented loops, so that 
\begin{equation}
\pi_1(S_g) = \langle a_i, b_i \, | \, [a_1, b_1] \cdots [a_g, b_g] \rangle, 
\end{equation}
where
$[a, b] = a b a^{-1} b^{-1}$ is the commutator. $H_1(S_g; \ZZ)$ can be
equipped with a symplectic form $\hat\iota$ for which 
$a_1,\ldots,a_g,b_1,\ldots,b_g$ form the standard symplectic basis, so that $\hat\iota$ is represented by the matrix 
\begin{equation}
  J_g =
  \begin{pmatrix}
    0 & I_g \\ -I_g & 0
  \end{pmatrix}.
\end{equation}

Any positive diffeomorphism $\psi : S_g \rightarrow S_g$ induces an automorphism
$\psi_\ast : H_1(S_g; \ZZ) \rightarrow H_1(S_g; \ZZ)$, which preserves $\hat\iota$,
and this automorphism does not depend on the isotopy class of $\psi$. This allows us to define the
symplectic representation of the mapping class group,
\begin{equation}
   \mcg  \rightarrow \Sp(2g, \ZZ) , \qquad \psi \mapsto M_\psi ,
\end{equation}
which associates with each $\psi \in \mcg$ a symplectic matrix $M_\psi$ representing $\psi_\ast$ in the basis
$(a_i, b_i)_{1 \leq i \leq g}$.  Here $\Sp(2g, \ZZ)$ denotes the set of $2g \times 2g$ matrices $M$
with integer coefficients satisfying $\trans M  J_g M = J_g$.

In order to define the affine mapping class group $\Lambda_g(\RR)$, we need to attach additional data to the surface $X$, in the form of a row vector $\vectheta\in\RR^{2g}$ representing the Aharonov-Bohm flux. 
We denote by $\widetilde\scrT_g=\scrT_g\times\RR^{2g}$ the extended Teichm\"uller space. 

A natural action of $\mcg$ on the extended Teichm\"uller space $\widetilde\scrT_g$ is given by
  \begin{equation}\label{MCGaction}
    \psi \cdot (X, \phi, \vectheta) = (X, \phi \circ \psi^{-1}, \vectheta M_\psi^{-1}).
  \end{equation}
  We furthermore define the $\RR^{2g}$ action on $\widetilde\scrT_g$ by translation of the vector
  component by
  \begin{equation}
    \veceta \cdot (X, \phi, \vectheta) = (X, \phi, \vectheta + \veceta) .
  \end{equation}
This suggests the definition of the semidirect product, the {\em affine mapping class group}
\begin{equation}\Lambda_g(\RR)= \mcg\ltimes \RR^{2g}\end{equation} with multiplication law
  \begin{equation}
    (\psi, \veceta) (\psi', \veceta') := (\psi \circ \psi', \veceta + \veceta' M_{\psi}^{-1}).
  \end{equation}
Note that $(\psi, \veceta) = (\id, \veceta)(\psi, \vecnull)$ and that $A_g(\RR)=\{\id\}\ltimes \RR^{2g}$ is a normal subgroup of $\Lambda_g(\RR)$. The action of $\Lambda_g(\RR)$ on $\widetilde\scrT_g$ is thus
 \begin{equation}\label{2.9}
    (\psi, \veceta) \cdot (X, \phi, \vectheta) := (X, \phi \circ \psi^{-1}, \veceta + \vectheta M_{\psi}^{-1}).
  \end{equation}

Define the subgroup $\Lambda_g(\ZZ)=\mcg\ltimes \ZZ^{2g}$ and its normal subgroup $A_g(\ZZ)=\{\id\}\ltimes \ZZ^{2g}$. Furthermore, set $\TT^{2g}=\RR^{2g}/\ZZ^{2g}$. We then have the following.

\begin{prop}\label{prop:bundle}
The extended moduli space $\widetilde\scrM_g:=\Lambda_g(\ZZ)\backslash\widetilde\scrT_g$ is a non-trivial fiber bundle over $\scrM_g$ with fiber $\TT^{2g}$. 
\end{prop}

\begin{proof}
We have that $A_g(\ZZ)\backslash\widetilde\scrT_g = \scrT_g \times \TT^{2g}$. That is, we can view $A_g(\ZZ)\backslash\widetilde\scrT_g$ as a trivial fiber bundle over $\scrT_g$ with fiber $\TT^{2g}$. This proves that $\widetilde\scrM_g$ is a torus bundle over $\scrM_g$. By the normality of $A_g(\ZZ)$, the action of the mapping class group defined in \eqref{MCGaction} induces an action on $\scrT_g \times \TT^{2g}$ given by
  \begin{equation}\label{MCGaction2}
    \psi \cdot (X, \phi, \vectheta+\ZZ^{2g}) = (X, \phi \circ \psi^{-1}, \vectheta M_\psi^{-1}+\ZZ^{2g}).
  \end{equation}
Note that $\psi\mapsto M_\psi\in\Sp(2g,\ZZ)$ acts non-trivially on $\TT^{2g}$, hence the bundle is non-trivial. 
\end{proof}

Formula \eqref{2.9} shows that if $\scrF_g$ is a fundamental domain for the action of $\mcg$ in $\scrT_g$, then $\widetilde\scrF_g=\scrF_g\times (-\tfrac12,\tfrac12]^{2g}$ is a fundamental domain for the action of $\Lambda_g(\ZZ)$ in $\widetilde\scrT_g$.

We furthermore note that $\scrM_g$ can be represented as $\Lambda_g(\RR)\backslash\widetilde\scrT_g$, and the base projection $\Pi:\widetilde\scrM_g\to\scrM_g$ is then given by $\Lambda_g(\ZZ)\cdot (X, \phi, \vectheta) \mapsto \Lambda_g(\RR)\cdot (X, \phi, \RR^{2g})$.

\section{Invariant measures and Fourier transforms}
\label{secInvariant}

We denote by $\Vol_g^{\mathrm{WP}}$ the Weil-Petersson measure on $\scrT_g$. Since it is $\mcg$-invariant, it descends to a measure $\mu_g^{\mathrm{WP}}$ on $\scrM_g$. This measure is finite and we normalise it to be a probability measure. That is, for any Borel set $\scrB\subset\scrM_g$,
\begin{equation}
\mu_g^{\mathrm{WP}}(\scrB) = \frac{\Vol_g^{\mathrm{WP}}(\scrF_g\cap \pi^{-1}\scrB)}{\Vol_g^{\mathrm{WP}}(\scrF_g)},
\end{equation}
where $\pi$ is the canonical projection $\scrT_g\to\scrM_g$ and $\scrF_g$ is a measurable fundamental domain of $\mcg$ in $\scrT_g$.

Let us now lift this measure to natural measures on the extended Teichm\"uller and moduli spaces. For $q\in\NN\cup\{\infty\}$, we define the measure $\Vol_{g,q}$ on $\widetilde\scrT_g$ by
\begin{equation}\label{Volgq}
\Vol_{g,q} = \Vol_g^{\mathrm{WP}} \times \nu_q^{2g}
\end{equation}
where $\nu_\infty^{k}=\Vol_{\RR^{k}}$ is the standard Lebesgue measure on $\RR^k$ and otherwise
\begin{equation}\label{nudef}
\nu_q^{k} = \frac{1}{q^{k}} \sum_{\vecm\in\ZZ^{k}} \delta_{q^{-1} \vecm} .
\end{equation}
Here $\delta_\veceta$ denotes the Dirac mass at $\veceta$. The point of these definitions is that
each $\nu_q^{2g}$ is invariant under the action of $\Sp(2g,\ZZ)$ (by linear transformations) and translation invariant by a subgroup containing $\ZZ^{2g}$. Therefore they descend to natural $\Sp(2g,\ZZ)$-invariant probability measures on $\TT^{2g}$, which for simplicity we will also denote by $\nu_q^{2g}$. Explicitly, $\nu_\infty^{k} = \Vol_{\TT^{k}}$ the Haar-Lebesgue probability measure and, for $q\in\NN$,
\begin{equation}\label{nudefq}
\nu_q^{k} = \frac{1}{q^{k}} \sum_{\vecm\in\ZZ_q^{k}} \delta_{q^{-1} \vecm} 
\end{equation}
with the shorthand $\ZZ_q=\ZZ/q\ZZ$.

With these assumptions, the measure $\Vol_{g,q}$ is invariant under the extended mapping class group $\Lambda_g(\ZZ)$, hence descends, after the appropriate normalisation, to a probability measure $\mu_{g,q}$ defined on Borel sets $\scrB\subset\widetilde\scrM_g$ by
\begin{equation}\label{mugq}
\mu_{g,q}(\scrB) = \frac{\Vol_{g,q}(\widetilde\scrF_g\cap \pi^{-1}\scrB)}{\Vol_{g,q}(\widetilde\scrF_g)}
\end{equation}
where $\pi$ now denotes the canonical projection
$\widetilde\scrT_{g} \rightarrow \widetilde\scrM_{g}$.  Note that
$\Vol_{g,q}(\widetilde\scrF_g)=\Vol_{g}(\scrF_g)$, since $\nu_q^{2g}(\TT^{2g})=1$.

In the following we denote by $\trans\ZZ_q^{2g}$ the set of integer column vectors $\trans\vecm$
with $2g$ coefficients in $\ZZ_q=\ZZ/q\ZZ$.  For functions $F:\scrT_g\times\TT^{2g}\to\CC$ so that
$F(X,\phi,\,\cdot\,)\in\L^1(\TT^{2g},\nu_q^{2g})$, we define the {\em fiber Fourier transform}
$\widehat F_q: \scrT_g\times\trans\ZZ_q^{2g}\to\CC$ by
\begin{equation}\label{sixfive}
\widehat F_q(X,\phi,\trans\vecm) = \int_{\TT^{2g}} F(X,\phi,\vectheta) \, e(-\vectheta \trans\vecm)\,d\nu_q^{2g}(\vectheta) ,
\end{equation}
with $e(x)=\mathrm{e}^{2 i \pi x}$. 
This is the standard Fourier transform for $q=\infty$ (put $\ZZ_q=\ZZ$), and the discrete Fourier transform if $q\in\NN$. Of particular interest is the zeroth Fourier coefficient, which we denote by
\begin{equation}
\breve F_q(X,\phi) = \widehat F_q(X,\phi,\trans\vecnull) .
\end{equation}

We define the dual action of $\mcg$ on $\scrT_g\times\trans\ZZ^{2g}$ by
\begin{equation}
\psi \star (X, \phi, \trans\vecm) = (X, \phi \circ \psi^{-1}, M_\psi \trans\vecm).
\end{equation}
Then
\begin{equation}\label{FTstar}
\widehat F_q(\psi\star (X,\phi,\trans\vecm)) = \int_{\TT^{2g}} F(\psi\cdot (X,\phi,\vectheta)) \, e(-\vectheta \trans\vecm)\,d\nu_q^{2g}(\vectheta) .
\end{equation}
Hence, if $F$ is invariant under $\mcg$, and thus can be viewed as a function on $\widetilde\scrM_g$, then 
\begin{equation}
\widehat F_q(\psi\star (X,\phi,\trans\vecm)) = \widehat F_q(X,\phi,\trans\vecm)
\end{equation}
and
\begin{equation}
\breve F_q(\psi\cdot (X,\phi)) = \breve F_q(X,\phi)
\end{equation}
for all $\psi\in\mcg$.

In the following, we identify $\mcg$-invariant functions of $\scrT_g$ with functions of $\scrM_g$, and $\Lambda_g(\ZZ)$-invariant functions of $\widetilde\scrT_g$ with functions of $\widetilde\scrM_g$, using the same notation. 

With this convention, we can define the fiber Fourier transform for $F\in\L^1(\widetilde\scrM_g,\mu_{g,q})$ by
\begin{equation}
\widehat F_q(X,\phi,\trans\vecm) = \int_{\Pi^{-1}(X,\phi)} F(X,\phi,\vectheta) \, e(-\vectheta \trans\vecm)\,d\nu_q^{2g}(\vectheta)
\end{equation}
where $\widehat F_q$ is now viewed as a function of the quotient $\widetilde\scrM_{g,q}^\star:=\mcg\backslash(\scrT_g\times\trans\ZZ_q^{2g})$ with respect to the $\star$-action. The zeroth Fourier coefficient $\breve F_q$ is identified with a function of $\scrM_g$. The above discussion in particular implies the following.

\begin{prop}\label{WPint}
If $F\in\L^1(\widetilde\scrM_g,\mu_{g,q})$ then $\breve F_q\in\L^1(\scrM_g,\mu_g^{\mathrm{WP}})$ with
\begin{equation}
\int_{\widetilde\scrM_g} F \,d\mu_{g,q} = \int_{\scrM_g} \breve F_g \,d\mu_g^{\mathrm{WP}}.
\end{equation}
\end{prop}

We will now apply the above formula to a class of functions that arise in the geometric analysis of twisted Laplacians.
The reduction to a standard integral over moduli space will then allow us to apply Mirzakhani's integration formulas.

\section{Length-holonomy functions}\label{secLength}

Given a closed non-contractible oriented loop $\gamma$ on $S_g$, let $\ell_\gamma(X,\phi)$ denote the length of the oriented closed geodesic $\phi(\gamma)$ on $X$, and $\Stab_\gamma$ the stabiliser of $\gamma$ in $\mcg$. We denote by $[\gamma]\in H_1(S_g;\mathbb{Z})$ the holonomy of $\gamma$ represented as a column vector written in the basis $(a_i, b_i)_{1 \leq i \leq g}$, as defined above.

Let $f:\RR \times\TT\to\CC$ be a smooth compactly supported function. We naturally
associate to $f$ and $\gamma$ a function $f_\gamma$ defined on $\scrT_g\times\TT^{2g}$ by
\begin{equation} \label{basic}
f_\gamma(X,\phi,\vectheta) = f(\ell_\gamma(X,\phi) ,\vectheta\, [\gamma]) .
\end{equation}
For $\psi\in\Stab_\gamma$, we have $\ell_\gamma(X,\phi\circ\psi^{-1}) = \ell_\gamma(X,\phi)$ and $\vectheta M_\psi^{-1}\, [\gamma]=\vectheta \,[\psi^{-1}(\gamma)]=\vectheta \,[\gamma]$, and therefore 
\begin{equation}
f_\gamma(\psi\cdot(X,\phi,\vectheta)) = f_\gamma(X,\phi,\vectheta).
\end{equation}
We now define the function $F_\gamma$ by
\begin{equation}
F_\gamma(X,\phi,\vectheta)  = \sum_{\psi\in\mcg/\Stab_\gamma} f_\gamma(\psi\cdot(X,\phi,\vectheta)) .
\end{equation}
This function is, by construction, $\mcg$-invariant, i.e., $F_\gamma(\psi\cdot(X,\phi,\vectheta)) = F_\gamma(X,\phi,\vectheta)$ for all $\psi\in\mcg$, and may therefore be viewed as a function on $\widetilde\scrM_g$. Its fiber Fourier transform is as follows.

\begin{prop}\label{prop003}
  For $\gamma$ as above, $q\leq\infty$, $\vecm\in\ZZ_q^{2g}$ and $f$ smooth with compact support, we
  have
\begin{equation}\label{FTimes}
\widehat F_{\gamma,q}(X,\phi,\trans\vecm)  = \sum_{\psi\in\mcg/\Stab_\gamma} \widehat f_{\gamma,q}(\psi\star(X,\phi,\trans\vecm)) ,
\end{equation}
where the fiber Fourier transform $\hat{f}_{\gamma,q}$ of $f_\gamma$ is equal to
\begin{equation} \label{FTimes2}
\widehat f_{\gamma,q}(X,\phi,\trans\vecm) = 
\sum_{n\in\ZZ_q} 
\widehat f_q(\ell_\gamma(X,\phi) , n) \;
\one(\trans\vecm = n [\gamma]\bmod q\trans\ZZ^{2g})
\end{equation}
with, for $(x, n)\in\RR\times\ZZ_q$,
\begin{equation}
\widehat f_q(x, n) = \int_\TT f(x,y) \, e(-ny) \, d\nu_q^1(y) .
\end{equation}
\end{prop}

\begin{proof}
In view of \eqref{sixfive}, we have
\begin{equation}\label{sixfive222}
\widehat F_q(X,\phi,\trans\vecm) = 
\sum_{\psi\in\mcg/\Stab_\gamma} \int_{\TT^{2g}}  f_\gamma(\psi\cdot(X,\phi,\vectheta))
\, e(-\vectheta \trans\vecm)\,d\nu_q^{2g}(\vectheta) .
\end{equation}
Now \eqref{FTstar} yields
\begin{equation}
\int_{\TT^{2g}}  f_\gamma(\psi\cdot(X,\phi,\vectheta))
\, e(-\vectheta \trans\vecm)\,d\nu_q^{2g}(\vectheta) = \widehat f_{\gamma,q}(\psi\star(X,\phi,\trans\vecm)),
\end{equation}
where
\begin{equation} \label{FTimes2221222}
\widehat f_{\gamma,q}(X,\phi,\trans\vecm) = 
\int_{\TT^{2g}}  f_\gamma(X,\phi,\vectheta)
\, e(-\vectheta \trans\vecm)\,d\nu_q^{2g}(\vectheta)
\end{equation}
is the fibre Fourier transform of $f_\gamma$. This establishes \eqref{FTimes}. 

Next, we expand the length-holonomy function in terms of the Fourier series of the test function $f$,
\begin{equation} \label{basic001}
f_\gamma(X,\phi,\vectheta) = f(\ell_\gamma(X,\phi) ,\vectheta\, [\gamma]) 
= \sum_{n\in\ZZ_q} 
\widehat f_q(\ell_\gamma(X,\phi) , n)\; e(n \vectheta\, [\gamma]) .
\end{equation}
Formula \eqref{FTimes2} then follows by inserting this expression in \eqref{FTimes2221222}
and using the identity
\begin{equation}
\int_{\TT^{2g}}  e(\vectheta \trans\vecs)\,d\nu_q^{2g}(\vectheta) =
\one(\vecs = \vecnull \bmod q\ZZ^{2g})
\end{equation}
with $\trans\vecs=n[\gamma]-\trans\vecm$.
\end{proof}

In the special case of the zeroth Fourier coefficient, the formula reads
\begin{equation}\label{Fsupsum}
  \breve F_{\gamma,q}(X,\phi)
  = \sum_{\psi\in\mcg/\Stab_\gamma}  \breve f_{\gamma,q}(\ell_\gamma(\psi\cdot(X,\phi)) 
\end{equation}
where 
\begin{equation}
  \label{Fsupsum2}
  \breve f_{\gamma,q}(x) = \sum_{n\in\ker_q[\gamma]} \widehat f_q(x , n),
  \qquad
  \ker_q[\gamma] = \{ n\in\ZZ_q : n [\gamma] = \trans \vecnull \bmod q\trans\ZZ^{2g} \}.
\end{equation}
In particular, for $q=\infty$, we have
$$
\ker_\infty[\gamma] = 
\begin{cases}
\{0\} & ([\gamma]\neq\trans\vecnull)\\
\ZZ  & ([\gamma]=\trans\vecnull) 
\end{cases}
$$
and
$$
\breve f_{\gamma,\infty}(x) =
\begin{cases}
\displaystyle\int_\TT f(x,y) \, d\nu_\infty^1(y) & ([\gamma]\neq\trans\vecnull)\\[10pt]
f(x,0)  & ([\gamma]=\trans\vecnull) .
\end{cases}
$$

Note that although the holonomy $[\psi(\gamma)]$ depends on the choice of $\psi\in\mcg$, the
condition $n[\psi(\gamma)]= \trans \vecnull \bmod q \trans\ZZ^{2g}$ is independent of that choice, since
the integer lattice $q \trans\mathbb{Z}^{2g}$ is stabilised by $\Sp(2g,\ZZ)$. Hence
$\ker_q[\gamma]=\ker_q[\psi(\gamma)]$ for $\psi\in\mcg$.  We conclude that
$\breve F_{\gamma,q}(X,\phi)$ is $\mcg$-invariant and therefore a well-defined function on the
moduli space $\scrM_g$. In fact, $\breve F_{\gamma,q}(X,\phi)$ belongs to the family of ``geometric
functions'' studied by Mirzakhani when $\gamma$ is simple \cite{Mirzakhani07,Mirzakhani13}, and by
Anantharaman and the second author for general loops \cite{AnantharamanMonk}.

We extend the above discussion to multi-loops $\vecgamma=(\gamma_1,\ldots,\gamma_k)$, with each
$\gamma_j$ a closed non-contractible oriented loop on the base surface $S_g$. We set
$\vecell_\vecgamma=(\ell_{\gamma_1},\ldots,\ell_{\gamma_k}) \in \mathbb{R}_{>0}^k$ and
\begin{equation}
\Stab_\vecgamma = \{\psi\in\mcg : \psi(\gamma_j) = \gamma_j \;\forall j=1,\ldots,k\}.
\end{equation}
%Recall that $\gamma_j$ is oriented and $\mcg$ orientation-preserving.
In the following, $[\vecgamma]=([\gamma_1],\ldots,[\gamma_k])\in H_1(S_g;\mathbb{Z})^k$ is
represented as a $(2g\times k)$ matrix.  For a compactly supported smooth function
$f:\RR^k\times\TT^k\to\CC$ and $(X,\phi,\vectheta) \in \scrT_{g} \times \mathbb{T}^{2g}$, define
\begin{equation}\label{under}
f_\vecgamma(X,\phi,\vectheta) = f(\vecell_\vecgamma(X,\phi) ,\vectheta\, [\vecgamma]) 
\end{equation}
and
\begin{equation} \label{Fsum}
F_\vecgamma(X,\phi,\vectheta)  = \sum_{\psi\in\mcg/\Stab_\vecgamma} f_\vecgamma(\psi\cdot(X,\phi,\vectheta)) .
\end{equation}
We have the same $\mcg$-invariance as above, i.e.,
$F_\vecgamma(\psi\cdot(X,\phi,\vectheta)) = F_\vecgamma(X,\phi,\vectheta)$ for all $\psi\in\mcg$,
and can therefore view $F_\vecgamma$ as a function on $\widetilde\scrM_g$. Its fiber Fourier
transform is given in the following proposition.

\begin{prop}\label{prop4}
For $\vecgamma$ a multi-loop as above, $q\leq\infty$, $\vecm\in\ZZ_q^{2g}$ and $f$ smooth with compact support, we have
\begin{equation}\label{multiFT}
\widehat F_{\vecgamma,q}(X,\phi,\trans\vecm)  = \sum_{\psi\in\mcg/\Stab_\vecgamma} \widehat f_{\vecgamma,q}(\psi\star(X,\phi,\trans\vecm)), 
\end{equation}
where the fiber Fourier transform $\widehat f_{\vecgamma,q}$ of $f_{\vecgamma}$ is equal to
\begin{equation} \label{multiFT2}
\widehat f_{\vecgamma,q}(X,\phi,\trans\vecm) = 
\sum_{\vecn\in\ZZ_q^k} 
\widehat f_q(\vecell_\vecgamma(X,\phi) , \vecn) \;
\one(\trans\vecm = [\vecgamma]\trans\vecn\bmod q\trans\ZZ^{2g})
\end{equation}
with, for $(\vecx,\vecn) \in \mathbb{R}^k \times \mathbb{Z}_q^k$,
\begin{equation}
\widehat f_q(\vecx, \vecn) = 
\int_{\TT^k} f(\vecx,\vecy) \, e(-\vecn\trans\vecy) \, d\nu_q^k(\vecy) .
\end{equation}
\end{prop}

\begin{proof}
This is analogous to the proof of Proposition \ref{prop003}.
\end{proof}

The zeroth Fourier coefficient
$\breve F_{\vecgamma,q}(X,\phi)=\widehat F_{\vecgamma,q}(X,\phi,\trans\vecnull)$ is again $\mcg$-invariant
and thus well-defined on $\scrM_g$. Specifying $\vecm = \vecnull$ in \eqref{multiFT}, \eqref{multiFT2} yields formulas similar to
\eqref{Fsupsum} and \eqref{Fsupsum2} in the case of multi-loops.

\section{Twisted Laplacians}\label{secTwisted}

We will now explain how the data $(X,\phi,\vectheta)$ is related to the Fuchsian group $\Gamma$ and its character $\chi_\vectheta$ used in the introduction to define the twisted Laplacian.

Given $\vectheta\in\RR^{2g}$, we define the unitary character
$\chi_\vectheta: \pi_1(S_g) \to\CC^\times$ by
\begin{equation}\label{keychar}
 \chi_\vectheta([\gamma]) = e(\vectheta\, [\gamma]),
\end{equation}
where $[\gamma] \in H_1(S_g;\mathbb{Z})$ is represented as a column vector written in the basis
$a_1,\ldots,a_g$, $b_1,\ldots,b_g$ of $H_1(S_g;\mathbb{Z})$.

Let $X = \Gamma(X) \backslash \HH$ be a compact hyperbolic surface and $\phi : S_g \rightarrow X$ a
marking. Let us pick a (fixed) basepoint $z_0$ on our base surface $S_g$. We use $\phi(z_0)$ as a
basepoint to define the fundamental group $\pi_1(X)$, and choose a lift $\tilde{z}$ of $\phi(z_0)$
in $\HH$. For any $\eta \in \Gamma(X)$, the geodesic arc from $\tilde z$ to $\eta(\tilde z)$ on
$\HH$ projects onto a geodesic loop on $X$, which we denote as $\eta_{z_0} \in \pi_1(X)$. Then, taking
the preimage by the marking naturally induces a group isomorphism
\begin{equation}
  \label{eq:phi_pi}
  \begin{split}
    \iota_{(X,\phi)} : \Gamma(X) & \rightarrow \pi_1(S_g) \\
    \eta & \mapsto \phi^{-1}(\eta_{z_0}).
  \end{split}
\end{equation}
We can now define the unitary character $\chi_{(X, \phi, \vectheta)}: \Gamma(X) \to\CC^\times$ by
\begin{equation}
  \chi_{(X, \phi, \vectheta)}(\eta)= \chi_\vectheta ([\iota_{(X,\phi)}(\eta)])
  = e(\vectheta [\phi^{-1}(\eta_{z_0})]).
\end{equation}
Note that, because the homology class is invariant by change of basepoint and homotopy,
$\chi_{(X,\phi,\vectheta)}$ is a well-defined function of
$(X,\phi,\vectheta) \in \widetilde{\mathcal{T}}_g$, and independent from our choice of basepoint
$z_0$. Furthermore, for $\psi\in\mcg$, we have
 \begin{equation}
    \chi_{\psi\cdot(X, \phi, \vectheta)}(\eta) 
    = e(\vectheta M_{\psi}^{-1} \, [\psi \circ \phi^{-1}(\eta_{z_0})])
    = \chi_{(X, \phi, \vectheta)}(\eta)
  \end{equation}
 and,  for $\vecm\in\ZZ^{2g}$, 
 \begin{equation}
    \chi_{(X, \phi, \vectheta+\vecm)}(\eta) = \chi_{(X, \phi, \vectheta)}(\eta) .
  \end{equation}
Therefore the character family $\chi=\chi_{(X, \phi, \vectheta)}$ is invariant under $\Lambda_g(\ZZ)$ and hence parametrised by the extended moduli space $\widetilde\scrM_g$. This in turn yields a $\widetilde\scrM_g$-family of twisted Laplacians $\Delta_{(X, \phi, \vectheta)}=\Delta_\chi$.

We will consider $\Delta_{(X, \phi, \vectheta)}$ as random operators by drawing $(X, \phi, \vectheta)$ with respect to the probability measures $\mu_{g,q}$ defined in \eqref{mugq}. For $q=1$, the character is trivial, and for $q=2$ real-valued. The following proposition states that for $q\geq 3$ the characters are real-valued with vanishing probability in the limit of large genus.

\begin{prop}
For every $(X,\phi)\in\scrT_g$, we have
\begin{equation}
\nu_q^{2g}\big\{ \vectheta\in\TT^{2g} : \chi_{(X,\phi,\vectheta)}^2 =\id  \big\} = 
\begin{cases}
0 & (q=\infty) \\
1/q^{2g} & (q\in2\NN+1) \\
(2/q)^{2g} & (q\in2\NN) .
\end{cases}
\end{equation}
\end{prop}

\begin{proof}
This follows straightforwardly from the observation that $\chi_{(X,\phi,\vectheta)}^2 =\id$ if and only if $\vectheta \in \{0,1/2\}^{2g}\bmod\ZZ^{2g}$. 
\end{proof}

This observation is consistent with Theorem \ref{thm1} which shows that, although the random ensemble with respect to the measure $\mu_{g,q}$ ($q\geq 3$) includes twisted Laplacians with GOE statistics, these do not spoil the GUE statistics of the full ensemble. 

\section{Main Lemmas}\label{secMain}

We now state the main lemmas required for the proof of Theorem \ref{thm1}. Recall that non-contractible oriented loops on $S_g$ are identified with oriented closed geodesics on $X$, which in turn correspond to conjugacy classes of non-trivial group elements; primitivity is preserved under these identifications.

\begin{lem}\label{lem1A}
For $f:\RR\to\CC$ smooth with compact support, $\ell_\gamma=\ell_\gamma(X,\phi)$,
\begin{equation} \label{lem1Aeq}
      \EE_{g,q} \Bigg[ \sum_{n=1}^{\infty} \sum_{\substack{\gamma\\ \mathrm{prim.}}} 
      \frac{\ell_\gamma\, f(n \ell_\gamma)\, e(\vectheta [\gamma] n)}{2 \sinh(n \ell_\gamma/2)} \Bigg]
      = 
      \begin{cases}
      \displaystyle \EE_g^{\mathrm{WP}} \bigg[\sum_{n=1}^{\infty} \sum_{\substack{\gamma\; \mathrm{prim.}\\ [\gamma]=\trans\vecnull}}
      \frac{\ell_\gamma f(n \ell_\gamma)}{2 \sinh(n \ell_\gamma/2)} \bigg] & (q=\infty) \\
      \displaystyle \EE_g^{\mathrm{WP}} \bigg[\sum_{n=1}^{\infty} \sum_{\substack{\gamma\\ \mathrm{prim.}}}
      \frac{\ell_\gamma f(n q^*_\gamma \ell_\gamma)}{2 \sinh(n q^*_\gamma \ell_\gamma/2)} \bigg] & (q\in\NN) 
      \end{cases}
  \end{equation}
 where each sum is over non-contractible oriented primitive loops and $q^*_\gamma=q/\gcd([\gamma],q)$.
\end{lem}

In analogy with the comment after equation \eqref{Fsupsum2}, we note that
$q^*_{\psi(\gamma)}=q^*_\gamma$ for every element $\psi$ of $\mcg$. This follows directly from the
fact that $M_\psi\in\Sp(2g,\ZZ)$ and therefore $\gcd(M_\psi [\gamma])=\gcd([\gamma])$.

Furthermore, we note that for $\gamma$ a non-contractible \emph{simple} loop, we have $[\gamma]=\trans\vecnull$ if and only if $\gamma$ is separating.

\begin{lem}\label{lem2A}
For $f:\RR\to\CC$ smooth  with compact support, $\ell_\gamma=\ell_\gamma(X,\phi)$,
\begin{multline}  \label{lem2Aeq}
      \EE_{g,q} \Bigg[ \bigg| \sum_{n=1}^{\infty} \sum_{\substack{\gamma\\ \mathrm{prim.}}} 
      \frac{\ell_\gamma\, f(n \ell_\gamma)\, e(\vectheta [\gamma] n)}{2 \sinh(n \ell_\gamma/2)}  \bigg|^2 \Bigg]
      \\ =
      \displaystyle   \EE_g^{\mathrm{WP}} \Bigg[ \sum_{n_1,n_2=1}^{\infty}
      \sum_{\substack{\gamma_1,\gamma_2 \;\mathrm{prim.}\\ n_1[\gamma_1]=n_2[\gamma_2] \\ \bmod q
          \trans \ZZ^{2g} }} 
      \frac{\ell_{\gamma_1} f(n_1  \ell_{\gamma_1})}{2 \sinh(n_1 \ell_{\gamma_1}/2)} \; \frac{\ell_{\gamma_2}\overline f(n_2 \ell_{\gamma_2})}{2 \sinh(n_2 \ell_{\gamma_2}/2)} \Bigg] ,
  \end{multline}
where each sum is over non-contractible oriented primitive loops; remove ``$\bmod q \trans\ZZ^{2g}$'' if $q=\infty$.
\end{lem}

\begin{proof}
  We divide the sums over $\gamma$ in \eqref{lem1Aeq} and over $\vecgamma=(\gamma_1,\gamma_2)$ in \eqref{lem2Aeq} into orbits under the mapping class
  group, i.e., write them as
  \begin{equation}
    \label{eq:lemA_orbit}
    \EE_{g,q} \Bigg[\sum_{\substack{\vecgamma\bmod\mcg \\ \text{prim.}}} F_\vecgamma(X,\phi,\vectheta)\Bigg]
  \end{equation}
  where the sum runs over representatives $\vecgamma$ of $\mcg$-orbits of $k$ primitive multi-loops,
  with $k=1$ for \eqref{lem1Aeq} and $k=2$ for \eqref{lem2Aeq}, and $F_\vecgamma$ is the
  length-holonomy function associated to the loop $\vecgamma$ and the functions
  \begin{equation}
    \tilde{f}(x,\theta) = \sum_{n=1}^{\infty} \frac{x f(nx) \, e(n\theta)}{2 \sinh (nx/2)} 
  \end{equation}
  for \eqref{lem1Aeq}, and
  \begin{equation}
    \tilde{f}(x_1,x_2,\theta_1,\theta_2)
    = \sum_{n_1,n_2=1}^{\infty} \frac{x_1 f(n_1x_1) \, e(n_1 \theta_1)}{2 \sinh (n_1x_1/2)}
    \frac{x_2 \overline{f}(n_2x_2) \, e(-n_2\theta_2)}{2 \sinh (n_2x_2/2)}
  \end{equation}
  for \eqref{lem2Aeq}.
  By Proposition \ref{WPint}, the $\mu_{g,q}$-average of each $F_\vecgamma$ is equal to the
  $\mu_g^{\mathrm{WP}}$-average of its zeroth Fourier coefficient, i.e. we can rewrite
  \eqref{eq:lemA_orbit} as 
  \begin{equation}
    \EE_g^{\mathrm{WP}}
    \Bigg[\sum_{\substack{\vecgamma\bmod\mcg \\ \text{prim.}}} \breve{F}_\vecgamma(X,\phi)\Bigg].
  \end{equation}
  Both lemmas now follow from Proposition \ref{prop4} applied to the Fourier coefficient $\vecm=\vecnull$.
\end{proof}

\section{Proof of Theorem \ref{thm1}}\label{secProof}

For $q\in\NN$ set
\begin{equation}\label{Ifq}
  I_{f,q}(L,\tau) =
  \frac{4}{L} \int_0^\infty
  \sum_{n=1}^\infty \hat f\left(\frac{n q x}{L}\right) \frac{\sinh^2(x/2)}{\sinh(nqx/2)} \cos(nq\tau x)\, dx,
\end{equation}
and for $q=\infty$, set $I_{f,q}(L,\tau)=0$.

\begin{prop}\label{mainprop1}
Under the assumptions for Theorem \ref{thm1}, we have 
\begin{equation}
\lim_{g\to\infty} \EE_{g,q} \Bigg[ \sum_{n=1}^{\infty} \sum_{\substack{\gamma\\ \mathrm{prim.}}} 
      \frac{\ell_\gamma\, \hat h_{L,\tau}(n \ell_\gamma)\, e(\vectheta [\gamma] n)}{2 \sinh(n \ell_\gamma/2)} \Bigg]
 = I_{f,q}(L,\tau) .
\end{equation}
\end{prop}

\begin{proof}
Assume first $q<\infty$. Lemma \ref{lem1A} reduces the problem to the $g\to\infty$ asymptotics of
\begin{equation}
\EE_g^{\mathrm{WP}} \Bigg[ \sum_{n=1}^{\infty} \sum_{\substack{\gamma\\ \mathrm{prim.}}}
      \frac{\ell_\gamma \hat h_{L,\tau}(n q^*_\gamma \ell_\gamma)}{2 \sinh(n q^*_\gamma \ell_\gamma/2)} \Bigg].
\end{equation}
Rudnick's case corresponds to the case $q=1$ (where $q^*_\gamma=1$), and the proof of the above
statement follows from the identical argument. The terms that are shown to vanish as $g\to\infty$
for $q=1$ \cite[§4.3]{Rudnick} dominate the same terms for $q\geq 2$, and therefore also these
vanish. The only surviving contribution comes from non-separating simple oriented loops. These loops
form a single $\mcg$-orbit, the orbit of a fixed loop $\gamma$ which can be chosen such that
$[\gamma]=\trans (1,0,\ldots,0)$. This means that $\gcd([\gamma],q)=1$ and therefore
$q_\gamma^*=q$. The evaluation of
\begin{equation}
  \lim_{g \to \infty}
\EE_g^{\mathrm{WP}} \Bigg[ \sum_{n=1}^{\infty} \sum_{\psi\in\mcg/\Stab_\gamma}
\frac{\ell_{\psi(\gamma)}
  \hat h_{L,\tau}(n q \ell_{\psi(\gamma)})}{2 \sinh(n q \ell_{\psi(\gamma)}/2)} \Bigg]
\end{equation}
follows directly from Mirzakhani's integration formula \cite{Mirzakhani07} and volume asymptotics
\cite{Mirzakhani13}; as done in~\cite[§4.3]{Rudnick}.

The case $q=\infty$ follows in the same way from Lemma \ref{lem1A}: for non-separating simple oriented loops $\gamma$ we have $[\gamma]\neq\trans\vecnull$ and thus all terms in \eqref{lem1Aeq} vanish in the limit $g\to\infty$.
\end{proof}

\begin{prop}\label{mainprop2}
Under the assumptions for Theorem \ref{thm1}, we have 
\begin{multline}\label{main-prop}
\lim_{g\to\infty}  \EE_{g,q} \Bigg[ \bigg| \sum_{n=1}^{\infty} \sum_{\substack{\gamma\\ \mathrm{prim.}}} 
      \frac{\ell_\gamma\, \hat h_{L,\tau}(n \ell_\gamma)\, e(\vectheta [\gamma] n)}{2 \sinh(n \ell_\gamma/2)} \bigg|^2\Bigg]
 \\
 = 
\begin{cases}
\Sigma_{\mathrm{GOE}}^2(f) +|I_{f,q}(L,\tau)|^2 +O(L^{-2}\log L) & (q=1,2)\\
\Sigma_{\mathrm{GUE}}^2(f) +|I_{f,q}(L,\tau)|^2 +O(L^{-2}\log L) & (3\leq q \leq \infty) .
\end{cases}
\end{multline}
\end{prop}

\begin{proof}
Let us first consider $q<\infty$. By Lemma \ref{lem2A}, we need to find the large $g$ asymptotics of
\begin{equation}\label{sixseven}
\EE_g^{\mathrm{WP}} \Bigg[ \sum_{n_1,n_2=1}^{\infty} \sum_{\substack{\gamma_1,\gamma_2 \;\mathrm{prim.}\\ n_1[\gamma_1]=n_2[\gamma_2] \\ \bmod q \trans\ZZ^{2g} }} 
      \frac{\ell_{\gamma_1} \hat h_{L,\tau}(n_1  \ell_{\gamma_1})}{2 \sinh(n_1 \ell_{\gamma_1}/2)} \; \frac{\ell_{\gamma_2}\overline{\hat h}_{L,\tau}(n_2 \ell_{\gamma_2})}{2 \sinh(n_2 \ell_{\gamma_2}/2)} \Bigg].
\end{equation}
As above, the lower order terms for $q=1$ dominate those for $q\geq 2$, and hence we can refer to
the estimates of \cite[§5.2, §6]{Rudnick}.

The only leading order terms come from pairs $(\gamma_1,\gamma_2)$ of simple non-separating oriented
loops in one of the following three constellations: (i)~$\gamma_1=\gamma_2$, (ii)
$\gamma_1=\gamma_2^{-1}$, or (iii) $\gamma_1\cap \gamma_2=\emptyset$.

In case (i), the loops are represented as a single orbit under $\mcg$ of a simple non-separating
oriented loop $\gamma$ with $[\gamma]= \trans(1,0,\ldots,0)$. The contribution of these terms to \eqref{sixseven} is thus
\begin{equation}\label{sixsevenB}
  \EE_g^{\mathrm{WP}} \Bigg[
  \sum_{\substack{ n_1,n_2=1 \\ n_1=n_2 \bmod q}}^{\infty}
  \sum_{\psi\in\mcg/\Stab_\gamma}  
      \frac{\ell_{\psi(\gamma)}^2 \hat h_{L,\tau}(n_1  \ell_{\psi(\gamma)}) \overline{\hat h}_{L,\tau}(n_2 \ell_{\psi(\gamma)})}{4 \sinh(n_1 \ell_{\psi(\gamma)}/2) \sinh(n_2 \ell_{\psi(\gamma)}/2)} \Bigg].
\end{equation}
Rudnick shows in \cite[Lemma 5.2]{Rudnick} that for $q=1$ terms with $n_1+n_2\geq 3$ contribute at
most $O(L^{-2}\log L)$, and hence the same is true for the contribution of those terms to
\eqref{sixsevenB}. The remaining terms are $n_1=n_2=1$, and is worked out in \cite[Lemma 5.2]{Rudnick}, 
\begin{equation}\label{sixsevenC}
\EE_g^{\mathrm{WP}} \Bigg[ \sum_{\psi\in\mcg/\Stab_\gamma}  
      \frac{\ell_{\psi(\gamma)}^2 |\hat h_{L,\tau}(\ell_{\psi(\gamma)})|^2}{4 \sinh^2(\ell_{\psi(\gamma)}/2)} 
\Bigg] = \int_\RR |x| \, |\hat f(x)|^2 dx +O(L^{-2})     .
\end{equation}
We here recognize the GUE variance
$\Sigma_{\mathrm{GUE}}^2(f) = \int_{\mathbb{R}} |x| |\hat{f}(x)|^2 \, d x$.  [Note here that our
summation is over oriented rather than the non-oriented loops in \cite{Rudnick}, which explains the
differences in the way factors of $2$ are distributed in some intermediate steps of the computation;
obviously yielding the same final results because there is a $2$-to-$1$ correspondance between oriented
and non-oriented loops.]

Case (ii) is analogous, again represented as a single orbit under $\mcg$ of a simple non-separating
oriented loop $\gamma$ with $[\gamma]= \trans(1,0,\ldots,0)$. Since
$\ell_{\psi(\gamma)}=\ell_{\psi(\gamma^{-1})}$, the contribution of these terms is now
\begin{equation}\label{sixsevenD}
  \EE_g^{\mathrm{WP}} \Bigg[
  \sum_{\substack{ n_1,n_2=1 \\ n_1=-n_2 \bmod q}}^{\infty}
  \sum_{\psi\in\mcg/\Stab_\gamma}  
      \frac{\ell_{\psi(\gamma)}^2 \hat h_{L,\tau}(n_1  \ell_{\psi(\gamma)}) \overline{\hat h}_{L,\tau}(n_2 \ell_{\psi(\gamma)})}{4 \sinh(n_1 \ell_{\psi(\gamma)}/2) \sinh(n_2 \ell_{\psi(\gamma)}/2)} \Bigg].
\end{equation}
As before, the $q=1$ case bounds above terms with $n_1+n_2\geq 3$ by $O(L^{-2}\log L)$. This leaves
the case $n_1=n_2=1$, which is compatible with the condition $n_1=-n_2 \bmod q$ only for $q=1$ or
$q=2$. In the latter case \eqref{sixsevenD} becomes \eqref{sixsevenC}, so that the sum of the two
first constellations yields a term $2 \, \Sigma_{\mathrm{GUE}}^2(f) = \Sigma_{\mathrm{GOE}}^2(f)$. We have
thus explained the GOE resp.~GUE term in \eqref{main-prop}.

In the remaining case (iii), the loops are generated by the $\mcg$-orbit of one
multi-loop $\vecgamma=(\gamma_1,\gamma_2)$ with representatives $\gamma_1, \gamma_2$ which we can
pick so that $[\gamma_1] = \trans(1,0,\ldots,0)$ and $[\gamma_2]= \trans(0,1,0,\ldots,0)$. This
means that the condition $n_1[\gamma_1]=n_2[\gamma_2] \bmod q \trans\ZZ^{2g}$ simplifies to
$n_1,n_2=0\bmod q$.  The contribution of these terms is therefore
\begin{equation}\label{sixsevenE}
  \EE_g^{\mathrm{WP}} \Bigg[
  \sum_{n_1,n_2=1}^{\infty}
  \sum_{\psi\in\mcg/\Stab_{\vecgamma}}  
  \frac{\ell_{\psi(\gamma_1)} \ell_{\psi(\gamma_2)}
    \hat h_{L,\tau}(n_1 q \ell_{\psi(\gamma_1)})
    \overline{\hat h}_{L,\tau}(n_2 q\ell_{\psi(\gamma_2)})}
  {4 \sinh(n_1 q \ell_{\psi(\gamma_1)}/2) \sinh(n_2 q \ell_{\psi(\gamma_2)}/2)}\Bigg] .
\end{equation}
The same argument as in \cite[Lemma 5.3]{Rudnick} can be used to show that the terms for $\gamma_1$
and $\gamma_2$ decouple for large genus $g$, i.e., the $g\to\infty$ limit of \eqref{sixsevenE} is
given by $|I_{f,q}(L,\tau)|^2$. 
[Recall once more that our summation is over oriented loops, which explains an extra factor of $4$ compared with \cite{Rudnick}.]

The above proof remains the same for $q=\infty$ by dropping ``mod $q$''; the contribution from case (iii) is trivially zero.
\end{proof}

\begin{proof}[Proof of Theorem \ref{thm1}]
We expand \eqref{Sigmagq} as
\begin{equation}\label{Sigmagq001}
\Sigma_{g,q}^2(L,\tau; f) = \EE_{g,q} \bigg[ \big| \counting^{\mathrm{osc}} \big|^2 \bigg] - \big| \EE_{g,q} \big[ \counting^{\mathrm{osc}}  \big] \big|^2,
\end{equation}
where
\begin{equation}
\counting^{\mathrm{osc}}  = \counting - (g-1) \int_\RR h_{L,\tau}(r) r \tanh(\pi r) \, dr
= \sum_{n=1}^{\infty} \sum_{\substack{\gamma\\ \mathrm{prim.}}} 
      \frac{\ell_\gamma\, \hat h_{L,\tau}(n \ell_\gamma)\, e(\vectheta [\gamma] n)}{2 \sinh(n \ell_\gamma/2)} .
\end{equation}
Here the second equality is Selberg's trace formula \eqref{STF}, where  the sum over $\gamma$
corresponds to primitive oriented closed geodesics. Propositions \ref{mainprop1} and \ref{mainprop2}
thus provide the large genus asymptotics of the expected value and variance of
$\counting^{\mathrm{osc}}$, and Theorem \ref{thm1} follows.
\end{proof}

\section{Dirac operators and GSE statistics}\label{secDirac}

The objective of this final section is to explain how the above construction can be extended to study random Dirac
operators, and in particular to prove GSE-statistics in the case of spin-systems with time-reversal
invariance. We refer the reader to \cite{BolteStiepan,MonkStan} for a more detailed introduction to Dirac operators on hyperbolic surfaces.

The Dirac operator on the hyperbolic plane $\mathbb{H}$ is given as the matrix-valued
operator
\begin{equation}
  \mathrm{D}
  = i
  \begin{pmatrix}
    0 & i y \partial_x + y \partial_y - \frac 12 \\
    - iy \partial_x + y \partial_y - \frac 12 & 0
  \end{pmatrix}
\end{equation}
acting on differentiable functions $\HH \rightarrow \CC^2$. We now consider a discrete co-compact subgroup $\overline\Gamma$ of $\SL(2,\RR)$, which is assumed to contain $-1$. It is the double cover of $\Gamma\subset\PSL(2,\RR)=\SL(2,\RR)/\{\pm 1\}$ used in the setting of Laplacians, with the covering map given by $\eta\mapsto\pm \eta$. Since $\Gamma$ is the fundamental group of a smooth hyperbolic surface, we have that $\overline\Gamma$ is strictly hyperbolic, i.e., the only elements with $|\tr\eta|\leq 2$ are $\eta=\pm 1$. The action of $\eta\in\SL(2,\RR)$ on $\HH\times\S^1$ (with $\S^1$ the unit circle in the complex plane) is defined as
\begin{equation}
(z,\e^{\i t}) \mapsto \left( \frac{az+b}{cz+d}, \frac{cz+d}{|cz+d|}\e^{\i t} \right),
\qquad 
\eta =  \begin{pmatrix}
  a & b \\ c & d 
\end{pmatrix} \in\SL(2,\RR).
\end{equation} 
In particular, $-1$ acts by $(z,\e^{\i t}) \mapsto ( z, -\e^{\i t})$.
For $\eta\in\overline\Gamma$ as above, set
\begin{equation}
j_\eta(z) = \frac{cz+d}{|c z+d|}, \qquad J_\eta(z) =
  \begin{pmatrix}
    j_\eta(z) & 0 \\
    0 & j_\eta(\bar z)
  \end{pmatrix} .
\end{equation}

We fix a unitary character
$\chi : \overline\Gamma \rightarrow \CC^\times$ such that $\chi(-1)=-1$. We consider the action of
$\mathrm{D}$ on the space of differentiable functions $\varphi:\HH \rightarrow \CC^2$ satisfying
\begin{equation}
  \varphi(\eta\cdot z)
  = \chi(\eta)\, J_\eta(z)\, \varphi(z)
  \qquad \text{for all } \eta \in \overline{\Gamma},
\end{equation}
and define the Dirac operator $\mathrm{D}_\chi$ for $X=\Gamma\backslash\HH$ to be the self-adjoint extension of the restriction of $\mathrm{D}$ to this space of automorphic forms.
  
If $\varphi=\begin{pmatrix} \varphi_1 \\ \varphi_2 \end{pmatrix}$ is an eigenform of $\mathrm{D}_\chi$ with eigenvalue $r$, then $\begin{pmatrix} \varphi_1 \\ -\varphi_2 \end{pmatrix}$ is an eigenform with eigenvalue $-r$ \cite[Lemma~1]{BolteStiepan}.
This so-called \emph{chiral symmetry} implies therefore that the spectrum of the operator $\mathrm{D}_\chi$ is symmetric
around~$0$. We denote by $0 \leq r_0\leq r_1\leq \ldots \to\infty $ its non-negative eigenvalues with multiplicity, including half of the eigenvalues that are equal to $0$. 

As in the case of the twisted Laplacian, this counting function can be studied by means of the trace formula \cite[Theorem~1]{BolteStiepan}
\begin{equation}\label{DiracSTF}
   \sum_{j=0}^\infty h(r_j) = (g-1) \int_\RR h(r)\, r \coth(\pi r)\, dr + 
      \sum_{n=1}^{\infty} \sum_{\substack{\{\eta\}\\ \tr\eta>2}} \chi(\eta)^n\,
      \frac{\ell_\eta \hat{h}(n \ell_\eta)}{2 \sinh(n \ell_\eta/2)}
    \end{equation}
    where the sum is, as in the classical Selberg trace formula, over $\overline\Gamma$-conjugacy classes of
    primitive hyperbolic elements, but now with $\overline\Gamma$ considered a subgroup of $\SL(2,\RR)$,
    and the summation restricted to the elements $\eta$ with $\tr\eta>2$. 

    As noted in \cite[Proposition~1]{BolteStiepan}, the sequence of $\lambda_j=r_j^2+\frac14$
    coincides with the eigenvalues of the Maa\ss-Laplacian of weight $1$, hence the above trace
    formula is in fact the same as the classical Selberg trace formula in that setting \cite[Theorem
    4.11]{Hejhal1}.

    Let us now construct random ensembles of Dirac operators
    $\mathrm{D}_{(X,\phi,\vectheta)}=\mathrm{D}_\chi$ with $(X,\phi,\vectheta)$ sampled from
    $\widetilde \scrM_g$ with respect to the probability measure $\mu_{g,q}$ defined in Section
    \ref{secInvariant}.

    Let
    $X = \Gamma(X) \backslash \HH$ be a compact hyperbolic surface and $\phi : S_g \rightarrow X$ a
    marking. We define, as set out above, the group $\overline{\Gamma}(X)$ as the double
    cover of $\Gamma(X)\subset\PSL(2,\RR)$,
    \begin{equation*}
      \overline{\Gamma}(X) = \{\eta \in \SL(2,\RR), \pm \eta \in \Gamma(X)\}.
    \end{equation*}
    We recall that the marking induces a group isomorphism
    $\iota_{(X,\phi)}:\Gamma(X) \rightarrow \pi_1(S_g)$, defined in~\eqref{eq:phi_pi}, and that we
    have fixed a family of fundamental loops $a_1, \ldots, a_g, b_1, \ldots, b_g$ on $S_g$.  For
    every $i \in \{1, \ldots, g\}$, the group elements $\iota_{(X,\phi)}^{-1}(a_i)$ and
    $\iota_{(X,\phi)}^{-1}(b_i) \in \Gamma(X)$ each have two representatives in
    $\overline{\Gamma}(X)$; we pick $\tilde{a}_i$ and $\tilde{b}_i \in \overline{\Gamma}(X)$ to be
    their respective representatives with positive trace.  Then, the elements $-1$,
    $\tilde{a}_1, \ldots, \tilde{a}_g$, $\tilde{b}_1, \ldots, \tilde{b}_g$ generates the group
    $\overline{\Gamma}(X)$. For $\vectheta \in \RR^{2g}$, we define the unitary character
    $\chi_{(X,\phi,\vectheta)}:\overline\Gamma(X)\to\CC^\times$ by
    \begin{equation}
      \chi_{(X,\phi,\vectheta)}(\tilde{a}_j) = \chi_\vectheta(a_j) = e(\theta_j), \qquad 
      \chi_{(X,\phi,\vectheta)}(\tilde{b}_j) = \chi_\vectheta(b_j) = e(\theta_{j+g}) , \qquad \chi_{(X,\phi,\vectheta)}(-1)=-1,
    \end{equation}
    where $\chi_\vectheta$ is the character of $\pi_1(S_g)$ defined in \eqref{keychar}.
    
For a general $\eta\in\overline\Gamma$ and its associated loop $\gamma\in\pi_1(S_g)$, we have
$\chi_{(X,\phi,\vectheta)}(\eta)=\pm \chi_\vectheta(\gamma)$.    
    We capture the
    sign in this relation by defining $\sigma_\gamma\in\{\pm1\}$ as
    \begin{equation}
      \sigma_\gamma=\sign(\tr\eta)\chi_{(X,\phi,\vectheta)}(\eta) \chi_\vectheta(\gamma)^{-1} .
    \end{equation}
    We observe that positive trace is not preserved under multiplication, and $\sigma_{\gamma}$ is
    therefore not multiplicative in $\gamma$.  We have, however, the following.
\begin{lem}
$\sigma_{\gamma^{-1}}=\sigma_\gamma$ and more generally $\sigma_{\gamma^n}=\sigma_\gamma^n$ for
$n\in\ZZ$.  
\end{lem}
\begin{proof}
  To see this, note that if $\tr\eta>0$ then we have $\tr\eta^n>0$ (this is trivial for $\eta=1$,
  and a fact for all $\eta\in\SL(2,\RR)$ with $\tr\eta>2$). Hence, by the multiplicativity of
  characters, we have
  $\chi_{(X,\phi,\vectheta)}(\eta)^n=\chi_{(X,\phi,\vectheta)}(\eta^n)=\sigma_{\gamma^n}
  \chi_\vectheta(\gamma^n)=\sigma_{\gamma^n} \chi_\vectheta(\gamma)^n$, which yields
  $\sigma_{\gamma^n}=\sigma_\gamma^n$.
\end{proof}

This lemma allows us to write the right hand side of the trace formula \eqref{DiracSTF} in terms of geometric (in place of algebraic) data. We have
\begin{equation}\label{DiracSTF001}
   \sum_{j=0}^\infty h(r_j) = (g-1) \int_\RR h(r)\, r \coth(\pi r)\, dr + 
      \sum_{n=1}^{\infty} \sum_{\substack{\gamma\\ \mathrm{prim.}}} 
      \frac{\sigma_\gamma^n \ell_\gamma \, \hat h(n \ell_\gamma)\, e(\vectheta [\gamma] n)}{2 \sinh(n \ell_\gamma/2)}  .
    \end{equation}

We now sample $\vectheta\in\TT^{2g}$ with respect to the probability measure $\nu_q^{2g}$ as defined in Section~\ref{secInvariant}. Since the range of $\chi_{(X,\phi,\vectheta)}$ contains $-1$, it is natural to consider only even values of $q$ as well as $q=\infty$. 

The case $q=2$ corresponds to real characters $\overline{\chi}=\chi$, which means $\mathrm{D}_\chi$
is symmetric under time-reversal and we have Kramer's degeneracy \cite[Lemma 2]{BolteStiepan}; that is, every eigenvalue
$r_j$ has even multiplicity. This motivates the definition of the
following reduced counting function taking into account the systematic degeneracy,
\begin{equation}\label{CFct}
  \counting^q =
  \begin{cases}
    \frac 12 \sum_{j=0}^\infty h_{L,\tau}(r_j) & (q=2) \\[5pt]
    \sum_{j=0}^\infty h_{L,\tau}(r_j) & (q \geq 4) .
  \end{cases}
\end{equation}
As in the case of the Laplacian, we are interested in the large $g$ asymptotics of the spectral variance of the Dirac operator $D_\chi$ on $X$ with $\chi=\chi_{(X,\phi,\vectheta)}$,
\begin{equation}
\widetilde\Sigma_{g,q}^2(L,\tau; f) = \EE_{g,q} \bigg[ \big| \counting^q - \EE_{g,q} [\counting^q] \big|^2 \bigg] .
\end{equation}
The following theorem shows that the spectral statistics of a Dirac operator are also consistent with the expected random matrix distributions, i.e., GSE statistics in the presence of time-reversal symmetry, and GUE otherwise.
\begin{thm}
  \label{thm:dirac}
  Fix $\tau>0$. For $f:\RR\to\CC$ even with smooth, compactly supported Fourier transform, we have, for $L\geq 2$, 
  \begin{equation}
    \lim_{g\to\infty} \widetilde\Sigma_{g,q}^2(L,\tau; f) 
    = 
    \begin{cases}
      \Sigma_{\mathrm{GSE}}^2(f) +O(L^{-2}\log L) & (q=2) \\
      \Sigma_{\mathrm{GUE}}^2(f) +O(L^{-2}\log L) & (q\in2\NN+2, \; q=\infty) ,
    \end{cases}
  \end{equation}
  with the random matrix density $\Sigma_{\mathrm{GUE}}^2(f)$ defined in Theorem \ref{thm1} and
  \begin{equation}
    \Sigma_{\mathrm{GSE}}^2(f) = \int_\RR \frac{|x|}{2} \, |\hat f(x)|^2 dx.
\end{equation}
\end{thm}

The proof of Theorem \ref{thm:dirac} mirrors that of Theorem \ref{thm1}, and the key steps are the following 
analogues of Propositions \ref{mainprop1} and \ref{mainprop2}.

%Given a non-contractible oriented loop $\gamma$ on $S_g$, there is a unique $\Gamma$-conjugacy class $\{\eta\}$ with $\eta \in p_{(X,\phi)}^{-1}(\gamma)$ and $\tr\eta>2$. We define $\epsilon_\gamma(\vectheta)=\chi_{(X,\phi,\vectheta)}(\eta)$ and note that $\epsilon_\gamma(\vectheta)= \sigma_{\gamma} \chi_\vectheta(\gamma)$ with $\sigma_\gamma\in\{\pm1\}$,
%and also $\epsilon_{\gamma^{-1}}(\vectheta)=\chi_{(X,\phi,\vectheta)}(\eta^{-1})=\overline{\chi_{(X,\phi,\vectheta)}(\eta)}=\overline{\epsilon_\gamma(\vectheta)}$ since $\tr\eta=\tr(\eta^{-1})$. Thus in particular $\sigma_{\gamma^{-1}}=\sigma_\gamma$.

\begin{prop}\label{mainprop1B}
Under the assumptions for Theorem \ref{thm:dirac}, we have 
\begin{equation}\label{expected007}
\lim_{g\to\infty} \EE_{g,q} \Bigg[ \sum_{n=1}^{\infty} \sum_{\substack{\gamma\\ \mathrm{prim.}}} 
      \frac{\sigma_\gamma^n \ell_\gamma \, \hat h_{L,\tau}(n \ell_\gamma)\, e(\vectheta [\gamma] n)}{2 \sinh(n \ell_\gamma/2)} \Bigg]
 = I_{f,q}(L,\tau) ,
\end{equation}
with $I_{f,q}(L,\tau)$ as in \eqref{Ifq}.
\end{prop}

\begin{proof}
  The same argument as in proof of Proposition \ref{mainprop1} tells us that only primitive loops
  with $\gcd[\gamma]=1$ contribute to the leading order term.  The analogous calculation as in Lemma
  \ref{lem1A} shows that, for the restriction of the sum over primitive geodesics $\gamma$ with
  $\gcd[\gamma]=1$, we have
\begin{equation} \label{lem1Aeq222}
      \EE_{g,q} \Bigg[ \sum_{n=1}^{\infty} \sum_{\substack{\gamma\\ \mathrm{prim.}\\ \gcd[\gamma]=1}} 
      \frac{\sigma_\gamma^n \, \ell_\gamma\, \hat h_{L,\tau}(n \ell_\gamma)\, e(\vectheta [\gamma] n)}{2 \sinh(n \ell_\gamma/2)} \Bigg]
      = 
      \begin{cases}
      0 & (q=\infty) \\
      \displaystyle \EE_g^{\mathrm{WP}} \bigg[\sum_{n=1}^{\infty} \sum_{\substack{\gamma\\ \mathrm{prim.}\\ \gcd[\gamma]=1}}
      \frac{\ell_\gamma \hat h_{L,\tau}(n q \ell_\gamma)}{2 \sinh(n q \ell_\gamma/2)} \bigg] & (q\in2\NN) .
      \end{cases}
  \end{equation}
  The key point here is that, for $q<\infty$, the condition $\gcd[\gamma]=1$ restricts the sum over $n$ to multiples
  of $q$ (as in Lemma \ref{lem1A}). These are even (since $q$ is), and hence
  $\sigma_\gamma^n=1$ for such $n$. For $q=\infty$, the average restricts the sum to $[\gamma]=\trans\vecnull$ (cf.~Lemma \ref{lem1A}), for which $\gcd[\gamma]=0$ and hence these terms do not contribute. We can now proceed as in the proof of
  Proposition \ref{mainprop1}.
\end{proof}
 
 \begin{prop}\label{mainprop2B}
Under the assumptions for Theorem \ref{thm:dirac}, we have 
\begin{multline}\label{main-prop2B}
\lim_{g\to\infty}  \EE_{g,q} \Bigg[ \bigg| \sum_{n=1}^{\infty} \sum_{\substack{\gamma\\ \mathrm{prim.}}} 
      \frac{\sigma_\gamma^n \ell_\gamma\, \hat h_{L,\tau}(n \ell_\gamma)\, e(\vectheta [\gamma] n)}{2 \sinh(n \ell_\gamma/2)} \bigg|^2\Bigg]
 \\
 = 
\begin{cases}
\Sigma_{\mathrm{GOE}}^2(f) +|I_{f,q}(L,\tau)|^2 +O(L^{-2}\log L) & (q=2)\\
\Sigma_{\mathrm{GUE}}^2(f) +|I_{f,q}(L,\tau)|^2 +O(L^{-2}\log L) & (q\in2\NN+2, \; q=\infty) .
\end{cases}
\end{multline}
\end{prop}

\begin{proof}
As in the proof of Proposition \ref{mainprop2}, the only leading order terms in the high genus limit come again from pairs $(\gamma_1,\gamma_2)$ of simple non-separating oriented loops, such that one of the following holds:
(i) $\gamma_1=\gamma_2$,
(ii) $\gamma_1=\gamma_2^{-1}$, or
(iii) $\gamma_1\cap \gamma_2=\emptyset$.
To evaluate these terms, we follow the argument used to prove Proposition \ref{mainprop2}, where we showed that in case (i) we have $n_1=n_2\bmod q$, in case (ii) $n_1=-n_2\bmod q$, and in case (iii) $n_1=n_2=0\bmod q$. Thus, if $q\in 2\NN$, we have for (i) and (ii) that
$\sigma_{\gamma_1}^{n_1} \sigma_{\gamma_2}^{n_2} = \sigma_{\gamma_1}^{n_1+n_2} = \sigma_{\gamma_1}^{n_1-n_2} = \sigma_{\gamma_1}^{q} =1$, and for (iii) 
$\sigma_{\gamma_1}^{n_1} \sigma_{\gamma_2}^{n_2} = \sigma_{\gamma_1}^{q} \sigma_{\gamma_2}^{q} = 1$. The case $q=\infty$ is analogous (drop the mod $q$).
This proves that the leading order terms in the high genus limit are exactly the same as in Proposition \ref{mainprop2}. 
\end{proof}

\begin{proof}[Proof of Theorem \ref{thm:dirac}] 
We apply the Selberg trace formula \eqref{DiracSTF001} for the Dirac operator with test function $h=h_{L,\tau}$ in order to convert the spectral variance to the geometric sums of Propositions \ref{mainprop1B} and \ref{mainprop2B}. Hence the claim is proved when $q=4,6,\ldots,\infty$. For the remaining case $q=2$, we observe that due to time-reversal symmetry ($\chi$ is real-valued) the counting function \eqref{CFct} is reduced by a factor of 2 to account for Kramer's degeneracy. Hence we should only take $\frac12$ of \eqref{expected007}, and $\frac14$ of \eqref{main-prop2B}. This indeed yields the formula for the GSE density, since $\Sigma_{\mathrm{GSE}}^2=\frac14 \Sigma_{\mathrm{GOE}}^2$.
\end{proof}

\end{document}